\documentclass[11pt]{amsart}
\usepackage{amsmath, amsthm, amssymb,xspace}
\usepackage[breaklinks=true]{hyperref}
\usepackage{amsmath,amscd}
\usepackage{enumerate}
\usepackage{xcolor}
\usepackage[margin=1.5in]{geometry}

\newtheorem{theorem}{Theorem}
\newtheorem{lemma}[theorem]{Lemma}
\newtheorem{prop}[theorem]{Proposition}

\newtheorem*{A}{ Theorem A}
\newtheorem*{B}{ Theorem B}
\newtheorem*{F}{ Theorem C}
\newtheorem*{proc}{Problem}

\theoremstyle{definition}

\newtheorem{remark}[theorem]{Remark}

\newtheorem*{proA}{Problem A}
\newtheorem*{proB}{Problem B}
\newtheorem*{conj}{Conjecture}
\newtheorem*{proofa}{Proof of Theorem A}
\newtheorem*{proofb}{Proof of Theorem B}

\newtheorem{eg}{Example}

\begin{document}

\title[Semi-commutants of Toeplitz Operators]{Semi-commutants of Toeplitz Operators on Fock-Sobolev space of Nonnegative orders }

\author{Jie Qin}
\address{School of Mathematics and Statistics, Chongqing Technology and Business University, 400067, China}
\email{qinjie24520@163.com}

\subjclass[2010]{Primary 47B35}
\keywords{Semi-commutants; Toeplitz operator; Fock-Sovolev sapce; Commuting problem}
\thanks{The author was supported by the National Natural
Science Foundation of China (11971125, 12071155).}

\begin{abstract}

We make a progress towards describing the semi-commutants of Toeplitz
operators on Fock-Sobolev space of nonnegative orders. We generalize the results in \cite{Bauer1,Qin}.
For the certain
symbol spaces, we obtain two Toeplitz operators can semi-commute only in the trivial cases, which is different from  what is known for the classical
Fock spaces.
As an application, we consider the  conjecture which was shown to
be false for Fock space in \cite{MA}. The main results of this paper say that  there is the fundamental difference
between the geometries of Fock and Fock-Sobolev spaces.
\end{abstract}

\maketitle

\section{Introduction}
 Let $\mathbb{C}$ denote the complex plane and  $dA$ be the area measure. For any fixed nonnegative integer $m$,  $L^2_m$ is the space of Lebesgue measurable functions $f$ on $\mathbb{C}$ such that the function $f(z)$ is in $L^2(\mathbb{C},\frac{|z|^{2m}e^{-|z|^2}}{\pi m!}dA(z)).$ It is well known that $L^2_m$ is a Hilbert space with the inner product $$<f,g>=\frac{1}{\pi m!}\int_{\mathbb{C}} f(z)\overline{g(z)}|z|^{2m}e^{-|z|^2}dA(z),\quad f,g\in L^2_m.$$

The Fock-Sobolev space $F^{2,m}$ consists of all entire functions $f$ in  $L^2_m$. Obviously, the Fock-Sobolev space $F^{2,m}$ is a closed subspace of the Hilbert space $L^2_m$. There is an orthogonal projection $P$ from $L^2_m$ onto $F^{2,m}$, which is given by
$$Pf(z)=\frac{1}{\pi m!}\int_{\mathbb{C}} f(w){K_m(z,w)}|w|^{2m}e^{-|w|^2}dA(w),\quad f\in F^{2,m},$$
where $$K_m(z,w)=\sum_{k=0}^\infty \frac{m!}{(k+m)!} (z\overline{w})^k=\frac{e^{z\overline{w}}-q_m(z\overline{w})}{(z\overline{w})^m}$$ is the reproducing kernel of Fock-Sobolev space $F^{2,m}.$ Here, $q_0=0$ and $q_m$ is the Taylor polynomial of $e^{z\overline{w}}$ of order $m-1$ for all $m\geq1$,  see \cite{Cho}.

Let $\mathcal{D}$ be the set of all finite linear combinations of kernel functions. Suppose $\varphi$ is a Lebesgue measurable function  on $\mathbb{C}$  that satisfies
\begin{equation}\label{welldown}
\int_{\mathbb{C}}|\varphi(w)||K_m(z,w)||w|^{2m}e^{-|w|^2}dA(w)<\infty.
\end{equation}
So we can densely define the Toeplitz operator with the symbol $\varphi$ on $F^{2,m}$ as follows:
 $$T_\varphi f=P(\varphi f).$$

Let $k_{m,z}(w)=k_m(w,z)$ be the normalization of the reproducing kernel $K_m(w,z).$ If $\varphi$ satisfies condition (\ref{welldown}), the Berezin transform of  $T_\varphi$ is $$\widetilde{T}_\varphi(z)=<T_\varphi k_m(w,z),k_m(w,z)>.$$
For the Fock space $F^2$, there is a Weyl operator $U_z$ on $F^2$ such that $U_{z} f(w)=f(w-z)k_{F, z}(w),$ where $k_{F, z}(w)$ is normalization of the reproducing kernel $F^2.$ By the property of $U_z$, the Berezin  transform of $f$ such that $$
\widetilde{f}(z)=\int_{\mathbb{C}} f(z \pm w) e^{-|w|^2}dA(w).
$$

For two entire functions $f$ and $g$, the semi-commutant of $T_f$ and $T_{\overline{g}}$ is given by
$$(T_f,T_{\overline{g}}]=T_{f\overline{g}}-T_{f}T_{\overline{g}}.$$ Many authors have studied the  semi-commutants(or commutants) on the Hardy space and Bergman space.
For these studies, refer to \cite{Axler,Brown,Lee,zheng}. The
problem  on Fock space  being still far from its solution, all the results
known so far  are restricted to certain subclasses of symbols; see, for example, \cite{Bauer1,Bauer2}. This commuting problem is left open.

Define $\epsilon(\mathbb{C})$ be the set of holomorphic functions $g$ on $\mathbb{C}$ such that
$|g(z)|e^{-c|z|}$ is essentially bounded on $\mathbb{C}$ for some $c>0.$
Let $\mathcal{P}$ be the space of all holomorphic polynomials on $\mathbb{C}$. Define $$\mathcal{A}_1=\bigg\{\sum_{j=1}^N p_j K_{a_j}: p_j\in \mathcal{P},a_j\in \mathbb{C}\quad\text{for}\quad j=1,\cdots,N\bigg\}.$$ In fact, if $f,g\in \epsilon(\mathbb{C})$ and $T_fT_{\overline{g}}=T_{f\overline{g}}$ on $F^2,$ then  $f,g\in \mathcal{A}_1$ by the elementary theory of the ordinary differential equations with constant coefficients. More explicitly,  for $f,g\in \mathcal{A}_{1},$ the berezin of $f\overline{g}$ must be of the form
$$
\sum_{i,j}\widetilde{\left\{p_i \bar{\eta_j} K_{a_i} \overline{K_{b_j}}\right\}}(z)=\sum_{i,j}e^{b_j \cdot \overline{a_i}} K_{a_i}(z) K_z(b_j) \eta_j^{*}\left(\partial_{z}+\overline{z}+\overline{a_i}\right) p(z+b_j),
$$
where $K_{a_j}$ is kernel of $a_j$ on Fock space, see Lemma 3.3 in \cite{Bauer1}.
Based on these facts, Bauer, Choe and Koo \cite{Bauer1} obtained
\begin{theorem}\label{A}\cite{Bauer1}
 Suppose $f,g\in \epsilon(\mathbb{C}),$ then the following statements are equivalent on $F^2$:\\
(a) $(T_f,T_{\overline{g}}]=0$.\\
(b)  $\widetilde{f\overline{g}}=f\overline{g}.$\\
(c)  Either (1) or (2) holds;

(1) $f$ or \ $g$ is constant.

(2)  There are finite collections  $\{a_j\}_{j=1}^N$  and  $\{b_l\}_{l=1}^M$ of distinct complex numbers such that
$$f \in Span \{K_{a_1},\cdots, K_{a_N}\},\quad \text{and} \quad g\in Span \{K_{b_1},\cdots, K_{b_M}\}$$
with \ $a_j\overline{b_l}= 2\lambda\pi \mathrm{i}$  for each  $j$ and $l$. Here, $\lambda$ is any integer.
\end{theorem}

 The Weyl operator $U_{z}$ is a key to consider the Operator Theory in $F^2$, the results of \cite{Bauer1,MA} are based on the property of $U_z.$
However, the translations do not have such good property on Fock-Sobolev space $F^{2,m}$. Namely, the Berezin transform of the function can almost never be computed explicitly. So we know very little about the the semi-commutativity of Toeplitz operator on Fock-Sobolev spaces.

So far, there are few results of semi-commuting Toeplitz operators  on the Fock-Sobolev space, and the symbol space is small.   Recall that $\mathcal{D}$ is the set of all finite linear combinations of kernel functions of Fock-Sobolev space. There is a  result, which can be found in our previous paper.
\begin{theorem}\label{B}\cite{Qin}
Suppose  $m$ is positive integer, the following conditions are equivalent for any two functions $f$ and $g$ in $\mathcal{D}$.\\
 (a) $T_{f}T_{\overline{g}}=T_{f\overline{g}}$ on $F^{2,m}$.\\
 (b)  $\widetilde{f\overline{g}}=f\overline{g}.$\\
 (c) At least one of  $f$ and $g$ is a constant.
\end{theorem}

 Let $\mathcal{D}_1$ be the set of all finite linear combinations of kernel function of Fock space.
There are two natural problems, which arise from Theorem \ref{A} and \ref{B}.
\begin{proA}\label{C}
Suppose $f\in \mathcal{D}$ and $g\in \mathcal{D}_1$. Determine the $f$ and $g$ for which the  semi-commutant  $(T_f,T_{\overline{g}}]$ is zero on $F^{2,m}$.
\end{proA}

\begin{proB}\label{D}
Suppose $f$ and $g$ are functions in $\mathcal{A}_{1}$. What is the relationship between their symbols when $(T_f,T_{\overline{g}}]=0$ on $F^{2,m}?$
\end{proB}

If $K_m(z,A)\in \mathcal{A}_{1}$ and $A,m\neq0,$ then
$$K_m(z,A)=\sum_{i=1}^N  p_i(z)e^{A_iz}.$$
 It follows that
$$0=\sum_{i=1}^N (z\overline{A})^mp_i(z)e^{A_iz}+m!q_m(z\overline{A})-m!e^{z\overline{A}}.$$
Clearly, $\sum_{j}(z\overline{A})^mp_j(z)-m!\not\equiv0$ for $A_j=\overline{A}$.
Note that $e^{A_iz}$ form a linearly independent set over the polynomials if $\{A_i\}_{i}$  is a finite collection of distinct non-zero complex numbers. Then  $\{e^{A_1z},\cdots,e^{A_Nz}, 1\}$ is linearly independent over the polynomials.
So, $\mathcal{D}\cap \mathcal{A}_{1}=\mathbb{C}$ if $m\neq0.$

 The purpose
of the current paper is to explore the two problems. In order to answer the first problem, we define the following symbol space $\mathfrak{D}$. The symbol space $\mathfrak{D}$ is given by
$$\mathfrak{D}=\mathcal{D}\cup \mathcal{D}_1 \cup \mathcal{P}.$$ We obtain  that there
are actually only trivial functions $f\in \mathcal{D}$ and $g\in \mathcal{D}_1$ such that
$(T_f,T_{\overline{g}}]=0$.  Precisely, the first main result is stated as follows.
\begin{A}
Let m be a positive integer, and suppose $f,g\in \mathfrak{D}$ . Then the following conditions are equivalent .\\
 (a) $(T_{f},T_{\overline{g}}]=0$ on $F^{2,m}$.\\
 (b)  $\widetilde{f\overline{g}}=f\overline{g}.$\\
 (c) At least one of  $f$ and $g$ is a constant.
\end{A}

 For  $f\in \mathcal{D}$ and $g\in \mathcal{D}_1$, one can see that
 \begin{equation}\label{1}
f(z)=\sum_{i=1}^{N} a_{i} K_m(z, A_{i}), \quad g(z)=\sum_{j=1}^{M} b_{j} e^{B_j z}.
\end{equation}
The main obstacle of the proof of the Theorem A is to calculate the semi-commutant, although $f$ and $g$ can be written in the form in (\ref{1}).  More specifically, the idea of the proof of the Theorem is to solve the following equations of $a_k$, $A_k$, $b_j$ and $B_j$:
$$\sum_{i=1}^{N}\sum_{j=1}^{M} a_{i} \overline{b_{j}}\bigg(T_{K_m(z, A_{i})\overline{e^{B_j z}}}-T_{K_m(z, A_{i}) }T_{\overline{e^{B_j z}}}\bigg)z^l=0,$$
for any nonnegative integer $l$. The main tool for showing our first result is the above infinite number of the function equations.

As we stated earlier, there are nontrivial functions $f$ and $g$ in $\mathcal{A}_{1}$ so that
$(T_f,T_{\overline{g}}]=0$ on Fock space $F^2.$ For $f,g\in \mathcal{A}_{1},$ the answer to the second question surprised us as the operator equation $(T_f,T_{\overline{g}}]=0$ on $F^{2,m}$ only have trivial solutions if $m\neq0$. The result is different from what is known for the classical Fock spaces.  The following theorem is the second main result of the paper.
\begin{B}
Let m be a positive integer, and suppose $f,g\in \mathcal{A}_{1}$. Then the following conditions are equivalent .\\
 (a) $(T_{f},T_{\overline{g}}]=0$ on $F^{2,m}$.\\
 (b)  $\widetilde{f\overline{g}}=f\overline{g}.$\\
 (c) At least one of  $f$ and $g$ is a constant.
\end{B}
 The main tool of the proof of  Theorem B is the Berezin transform. In the case of $m=0$, our proof recover the result of semi-commuting Toeplitz operators in \cite{Bauer1}.

For  holomorphic functions $f$ and $g$, it is clear that $$
H_{\overline{f}}^{*} H_{\overline{g}}=T_{f \overline{g}}-T_{f} T_{\overline{g}}=(T_f,T_{\overline{g}}].
$$
In \cite{MA}, the authors showed that the following conjecture is false on Fock space.
\begin{conj}
Given holomorphic functions $f$ and $g,$ the Hankel product $H_{\overline{f}}^{*} H_{\overline{g}}$ is bounded if and only if the function
$$
D(f, g)(z)=\left[\widetilde{|f|^{2}}(z)-|\widetilde{f}(z)|^{2}\right]\left[\widetilde{|g|^{2}}(z)-|\widetilde{g}(z)|^{2}\right]
$$
is bounded.
\end{conj}

Let $\mathbb{N}$ denote the set of integer. We say that $a\in 2 \mathbb{N} \pi \mathrm{i}$ means that
$a= 2n\pi \mathrm{i}$ for any integer $n.$ As an application, we consider the above conjecture on Fock-Sobolev space. Our result here reduces to \cite{MA}.
\begin{F}
The conjecture is false on $F^{2,m}$. In fact, there are functions $f=e^{Az}$ and $g(z)=e^{Bz}$ with $A\overline{B}\in2 \mathbb{N} \pi \mathrm{i}/\{0\} $ such that  $(T_f,T_{\overline{g}}]=H_{\overline{f}}^{*} H_{\overline{g}}$ is bounded while $D(f,g)$ is unbounded on $\mathbb{C}$.
\end{F}

 We remark that Theorem B and Theorem C can be extend to the Fock-Sobolev spaces on $\mathbb{C}^n$. Throughout the paper, $m$ will be a fixed nonnegative integer. We say that $a\lesssim b$ if there is a constant $c$ independent of $a$ and $b$ such that $a\leq cb$, where $a$ and $b$ are nonnegative quantities. Similarly, we say $a\simeq b$ if $a\lesssim b$ and $b\gtrsim a$.

\section{The symbol space $\mathfrak{D}$}
The problem to be studied in this section is to determine the function $f,g\in\mathfrak{D}$ for which
the $(T_f,T_{\overline{g}}]=0$ on Fock-Sobolev spaces. We begin with the  special case.
\begin{lemma}\label{LL6}
Suppose $p$ is a holomorphic polynomial and $f$ is a holomorphic function. Then the following statements are equivalent on $F^{2,m}$.\newline
(a) $(T_{f},T_{\overline{p}}]=0$.\newline
(b) $\widetilde{f\overline{p}}=f\overline{p}$.\newline
(c) one of $p$ and $f$ is a constant.
\end{lemma}
\begin{proof}
 The equivalence $(a) \Leftrightarrow (b)$ holds by Theorem \ref{B}. The implication $(c)\Rightarrow (a)$ is clear. To prove $(a)\Rightarrow(c)$.
Since $(T_{f},T_{\overline{p}}]=0$, then $T_{f \overline{p}}=T_{f}T_{\overline{p}}.$ One can see that
$$T_{f(z) \overline{p(z)}}1=T_{f(z)}T_{\overline{p(z)}}1=\overline{p(0)}f(z).$$
It follows that
$$T_{ \overline{p(z)}}f(z)=\overline{p(0)}f(z).$$
This shows that  \begin{equation}\label{E43}
T_{ \overline{p(z)}-\overline{p(0)}}f(z)=0.
\end{equation}

Let $$p(z)=\sum_{i=0}^{N_1}a_iz^i,\quad f(z)=\sum_{j=0}^\infty b_l z^l,$$
where $N_1$ is a finite positive integer. Now, we assume $f$ is not a polynomial. From (\ref{E43}) we have
\begin{equation}\label{E7}
0=\sum_{i=1}^{N_1} \overline{a_i} \sum_{l=i}^\infty b_l \frac{(l+m)!}{(l-i+m)!}z^{l-i}.
\end{equation}
From (\ref{E7}), we see that
$$0=\sum_{i=1}^{N_1} \overline{a_i} b_{\lambda+i} \frac{(\lambda+i+m)!}{(\lambda+m)!} $$ for all
$\lambda\geq0.$ Since $f$ is not a polynomial, then are distinct numbers
$$\lambda_1,\cdots, \lambda_{N_1}$$ so that
 the determinant of  the matrix $B$ is non-zero. Here
$$B=\left(
  \begin{array}{cccc}
    b_{\lambda_1+1} \frac{(\lambda_1+1+m)!}{(\lambda_1+m)!} & b_{\lambda_1+2}  \frac{(\lambda_1+2+m)!}{(\lambda_1+m)!} & \cdots & b_{\lambda_1+N_1}\frac{(\lambda_1+N_1+m)!}{(\lambda_1+m)!} \\
   b_{\lambda_2+1} \frac{(\lambda_2+1+m)!}{(\lambda_2+m)!}& b_{\lambda_2+2} \frac{(\lambda_2+2+m)!}{(\lambda_2+m)!} & \cdots & b_{\lambda_2+N_1} \frac{(\lambda_2+N_1+m)!}{(\lambda_2+m)!}\\
    \vdots & \vdots & \vdots & \vdots \\
    b_{\lambda_{N_1}+1}\frac{(\lambda_{N_1}+1+m)!}{(\lambda_{N_1}+m)!} & \cdots & \cdots & b_{\lambda_{N_1}+N_1} \frac{(\lambda_{N_1}+N_1+m)!}{(\lambda_{N_1}+m)!}\\
  \end{array}
\right).$$
This can be done, because $b_{\lambda_k+i}$ is the coefficient of $z^{\lambda_k+i}.$ Moreover,
 $BX=0,$ where $X=(a_1,\cdots,a_{N_1})^{T}$. Since $|B|\neq0,$ then we have $a_i=0$ for  all $1\leq i \leq  N_2.$ This shows that $p$ is a constant if $f$ isn't a polynomial.

In fact, one of the $p$ and $f$ is a constant function if $f$ is a  polynomial by Theorem 9 in \cite{Qin}.  This is the desired result.
\end{proof}
\begin{eg}
First, we find ${\lambda_1}$ such that $b_{\lambda_1+1}\neq0$, and then we choose ${\lambda_2}$  that satisfies ${\lambda_1}\neq {\lambda_2}$, $b_{\lambda_2+1}=0$ and $b_{\lambda_2+2}\neq0$. By this way,
we can obtain the special upper-triangular matrix
$$B=\left(
  \begin{array}{cccc}
    b_{\lambda_1+1} \frac{(\lambda_1+1+m)!}{(\lambda_1+m)!} & b_{\lambda_1+2}  \frac{(\lambda_1+2+m)!}{(\lambda_1+m)!} & \cdots & b_{\lambda_1+N_1}\frac{(\lambda_1+N_1+m)!}{(\lambda_1+m)!} \\
  0& b_{\lambda_2+2} \frac{(\lambda_2+2+m)!}{(\lambda_2+m)!} & \cdots & b_{\lambda_2+N_1} \frac{(\lambda_2+N_1+m)!}{(\lambda_2+m)!}\\
    \vdots & \vdots & \vdots & \vdots \\0& \cdots & \cdots & b_{\lambda_{N_1}+N_1} \frac{(\lambda_{N_1}+N_1+m)!}{(\lambda_{N_1}+m)!}\\
  \end{array}
\right).$$
\end{eg}
 The above Lemma shows the Berezin transform of a polynomial is again a polynomial. The following lemma has been considered in \cite{Qin} by direct calculation. Now, we prove it by the Berezin Transform.
\begin{lemma}\label{LL3}
Suppose $f(z)=e^{Az}$ and $g(z)=e^{Bz},$ then $T_{f\overline{g}}=T_{f}T_{\overline{g}}$ if and only if $A\overline{B}=0$.
\end{lemma}
\begin{proof}
If $AB=0$, then $T_{f\overline{g}}=T_{f}T_{\overline{g}}$. Conversely,
assume $AB\neq0,$ we define $$Q(z,\overline{z})=\frac{1}{m!\pi}\int_{\mathbb{C}} e^{Aw+\overline{B} \overline{w}}|K_m(z,w)|^2|w|^{2m} e^{-|w|^2}dA(w).$$
Then
\begin{align*}
Q(z,\overline{z})&=\frac{1}{m!\pi}\int_{\mathbb{C}} e^{Aw+\overline{B} \overline{w}}|K_m(z,w)|^2|w|^{2m} e^{-|w|^2}dA(w)\\
&=\frac{m!}{\pi|z|^{2m}}\int_{\mathbb{C}} e^{Aw+\overline{B} \overline{w}}[e^{\overline{z}w}-q_m(\overline{z}w)][e^{\overline{w}z}-q_m(\overline{w}z)] e^{-|w|^2}dA(w)\\
&=\frac{m!}{\pi|z|^{2m}}\int_{\mathbb{C}} e^{Aw+\overline{B} \overline{w}}
\bigg\{e^{\overline{z}w}e^{\overline{w}z}- q_m(\overline{w}z)e^{\overline{z}w}\\
&\quad-q_m(\overline{z}w)
e^{\overline{w}z}+q_m(\overline{z}w)q_m(\overline{w}z)\bigg\}e^{-|w|^2}dA(w).
\mathfrak{}\end{align*}
Define $$p(z,\overline{z})=\frac{1}{\pi}\int_{\mathbb{C}} e^{Aw+\overline{B} \overline{w}}q_m(\overline{z}w)q_m(\overline{w}z)e^{-|w|^2}dA(w),$$
note that $p$ is polynomial in $z$ and $\overline{z}$. In fact, if $A=0,$ then
\begin{align*}
&\quad\frac{1}{\pi}\int_{\mathbb{C}} e^{\overline{B} \overline{w}}q_m(\overline{z}w)q_m(\overline{w}z)e^{-|w|^2}dA(w)\\
&=\sum_{k=0}^{m-1} \sum_{l=0}^{k}\frac{1}{\pi}\int_{\mathbb{C}}
\frac{(\overline{z}w)^k}{k!}\frac{\overline{B}^lz^{k-l}\overline{w}^{k}}{l!(k-l)!}e^{-|w|^2}dA(w)\\
&=\sum_{k=0}^{m-1} \sum_{l=0}^{k} \frac{\overline{z}^k\overline{B}^lz^{k-l}}{l!(k-l)!}\\
&=\sum_{k=0}^{m-1}  \frac{\overline{z}^k}{k!}\sum_{l=0}^{k} \frac{k!\overline{B}^lz^{k-l}}{l!(k-l)!}\\
&=q_m[\overline{z}(\overline{B}+z)].
\end{align*}

By a simple computation,
\begin{equation}\label{R1}
\left\{
  \begin{array}{ll}
   \frac{1}{\pi} \int_{\mathbb{C}} e^{Aw+\overline{B} \overline{w}}
e^{\overline{z}w+\overline{w}z}e^{-|w|^2}dA(w)=e^{(A+\overline{z})(\overline{B}+z)}; \\
   \frac{1}{\pi} \int_{\mathbb{C}} e^{Aw+\overline{B} \overline{w}}
q_m(\overline{w}z)e^{\overline{z}w}e^{-|w|^2}dA(w)=e^{\overline{B}(A+\overline{z})}
q_m(zA+|z|^2) ;\\
   \frac{1}{\pi} \int_{\mathbb{C}} e^{Aw+\overline{B} \overline{w}}
q_m(\overline{z}w)
e^{\overline{w}z}e^{-|w|^2}dA(w)=e^{A(\overline{B}+z)}
q_m(\overline{z}\overline{B}+|z|^2).
  \end{array}
\right.
\end{equation}
Since $T_{f\overline{g}}=T_{f}T_{\overline{g}}$, then $\widetilde{f\overline{g}}=f\overline{g}.$ It follows that  \begin{equation}\label{E20}
m!\frac{e^{|z|^2}-q_m(|z|^2)}{|z|^{2m}}e^{Az+\overline{B}\overline{z}}
=Q(z,\overline{z}).
\end{equation} By (\ref{E20}) and Lemma 2.9 in \cite{Bauer1},
$$m!\frac{e^{z\overline{\eta}}-q_m(z\overline{\eta})}{(z\overline{\eta})^{m}}
=e^{-Az-\overline{B}
\overline{\eta}}Q(z,\overline{\eta})$$ for all $z,\eta\in \mathbb{C}.$
By (\ref{R1}),
\begin{align}\label{E38}
 K_m(z,\eta)&=m!\frac{e^{z\overline{\eta}}-q_m(z\overline{\eta})}{(z\overline{\eta})^{m}}\nonumber\\
&=\frac{m!}{(z\overline{\eta})^{m}}\bigg\{e^{A\overline{B}+z\overline{\eta}}
-e^{A\overline{B}-Az}q_m(zA+z\overline{\eta})
\\
&\quad-e^{A\overline{B}-\overline{B}\overline{\eta}}
q_m(\overline{B}\overline{\eta}+z\overline{\eta})+e^{-Az-\overline{B}
\overline{\eta}}p(z,\overline{\eta})\bigg\}\nonumber,
\end{align}
which shows that $e^{A\overline{B}}=1$. It follows from (\ref{E38}) that
\begin{align*}
 K_m(z,\eta)&=\frac{m!}{(z\overline{\eta})^{m}}\bigg\{e^{z\overline{\eta}}
-e^{-Az}q_m(zA+z\overline{\eta})
\\
&\quad-e^{-\overline{B}\overline{\eta}}
q_m(\overline{B}\overline{\eta}+z\overline{\eta})+e^{-Az-\overline{B}
\overline{\eta}}p(z,\overline{\eta})\bigg\}.
\end{align*}

Now, we set $A\overline{B}\neq0$, then
$$e^{-Az}[e^{-\overline{B}
\overline{\eta}}p(z,\overline{\eta})-q_m(zA+z\overline{\eta})]$$  is not a polynomial in $z$.
We also have
$$e^{-\overline{B_j}
\overline{\eta}}[e^{-Az}p(z,\overline{\eta})-q_m(\overline{B}\overline{\eta}+z\overline{\eta})]$$ is not a polynomial in  $\overline{\eta}$. In fact, $\{e^{-Az},1\}$ is linearly independent over the polynomials. This implies that
$$q_m(z\overline{\eta})\neq e^{-Az}\bigg[q_m(zA+z\overline{\eta})
+e^{-\overline{B}
\overline{\eta}}p(z,\overline{\eta})\bigg]-e^{-\overline{B}\overline{\eta}}
q_m(\overline{B}\overline{\eta}+z\overline{\eta}),$$ which contradicts the form of $K_m(z,\eta).$ We thus have  $A\overline{B}=0$. This complete the proof.
\end{proof}

Recall that $\mathcal{D}_1$ is the set of all finite linear combinations of kernel function of Fock space. We remark the $\sum_{i,j}e^{A_iz+\overline{B_jz}}$ with $A_i\overline{B_j}\in 2\mathbb{N}\pi \mathrm{i}$  is fixed points of Berezin transform on Fock spaces. However, it is false on Fock-Sobolev space.
\begin{prop}\label{pp2}
Suppose $f,g\in \mathcal{D}_1$ and $m$ is a positive integer.
Then the following statements are equivalent:\newline
(a) $\widetilde{f\overline{g}}=f\overline{g}$ on $F^{2,m}$.\newline
(b) $\widetilde{e^{A_iz+\overline{B_j} \overline{z}}}=e^{A_iz+\overline{B_j} \overline{z}}$  for each $i$ and $j$.\newline
(c) one of $f$ and $g$ is a constant.
\end{prop}

\begin{proof}
Suppose $$f(z)=\sum_{i=1}^N a_i e^{A_iz},\quad g(z)=\sum_{j=1}^M b_j e^{B_j z}.$$
Now assume (a) hold. Let $$Q_{ij}(z,\overline{z})=\frac{1}{m!\pi}\int_{\mathbb{C}} e^{A_iw+\overline{B_j} \overline{w}}|K_m(z,w)|^2|w|^{2m} e^{-|w|^2}dA(w).$$
We define $Q_{ij}(z,\overline{z})=\frac{m!}{|z|^{2m}}\rho_{ij}(z,\overline{z})$, then, by (\ref{E20}),
\begin{equation}\label{E28}
0=\sum_{i}\sum_{j} a_i\overline{b_j}\bigg\{[e^{|z|^2}-q_m(|z|^2)]e^{A_iz+\overline{B_j}\overline{z}}- \rho_{ij}(z,\overline{z})\bigg\}.
\end{equation}
It follows from (\ref{E38}) that
\begin{align}\label{E19}
0&=\sum_{i}\sum_{j}  a_i\overline{b_j} \bigg([e^{|z|^2}-q_m(|z|^2)]e^{A_iz+\overline{B_j}\overline{z}}-e^{(A_i+\overline{z})(\overline{B_j}+z)}
\nonumber\\
&\quad+e^{\overline{B_j}(A_i+\overline{z})}
q_m(zA_i+|z|^2)+e^{A_i(\overline{B_j}+z)}
q_m(\overline{z}\overline{B_j}+|z|^2)-p_{ij}(z,\overline{z})\bigg),
\end{align}
where each $p_{ij}$ is a polynomial as before in Lemma \ref{LL3} .

Without loss of generality, assume $\{A_i\}_{i=1}^{M_1}$ and $\{B_j\}_{j=1}^{M_2}$ are
two finite collections of distinct non-zero complex numbers. Moreover, we also assume $a_ib_j\neq 0$ for all $1\leq i \leq M_1$ and $1\leq j \leq M_2.$ Comparing the "coefficients" of $e^{|z|^2}$ in (\ref{E19}), one can see that
\begin{align*}
0&=\sum_{i}\sum_{j}  a_i\overline{b_j}e^{A_iz+\overline{B_j}\overline{z}}- \sum_{i}\sum_{j}a_i\overline{b_j}e^{A_i \overline{B_j}+A_iz+\overline{B_j}\overline{z}}\\
&= \sum_{i}\sum_{j}  a_i\overline{b_j}e^{A_iz+\overline{B_j}\overline{z}}(1-e^{A_i\overline{B_j}}).
\end{align*}

Note that  $e^{A_iz+\overline{B_j}\overline{z}}$ from a linearly independent set over the polynomials, because $(A_i,B_j)'$s are all distinct.  This shows that $e^{A_i\overline{B_j}}=1$ for all $1\leq i \leq M_1$ and $1\leq j \leq M_2.$ Putting this in (\ref{E19}), (\ref{E19}) becomes
 \begin{align*}\label{E36}
0&=\sum_{i}\sum_{j}  a_i\overline{b_j} \bigg(-q_m(|z|^2)e^{A_iz+\overline{B_j}\overline{z}}
+e^{\overline{B_j}\overline{z}}
q_m(zA_i+|z|^2)\\
&\quad+e^{A_iz}
q_m(\overline{z}\overline{B_j}+|z|^2)-p_{ij}(z,\overline{z})\bigg)\nonumber.
\end{align*}

Clearly, $q_m(|z|^2)>0$ if $m\neq0.$ Note that $m=0$ implies the above  equation always holds, since  $e^{A_j\overline{B_j}}=1$ for each $i$ and $j$. Now, we assume $m\neq0$ and define
$$v_{ij}=e^{A_iz+\overline{B_j}\overline{z}}-\frac{e^{\overline{B_j}\overline{z}}
q_m(zA_i+|z|^2)}{q_m(|z|^2)}-\frac{e^{A_iz}
q_m(\overline{z}\overline{B_j}+|z|^2)}{q_m(|z|^2)}
+\frac{p_{ij}(z,\overline{z})}{q_m(|z|^2)}.$$
Recall that $e^{A_iz+\overline{B_j}\overline{z}}$ form linearly independent set over the polynomials. In fact, for fixed $i,j,k,l$, we have  $\{e^{A_iz+\overline{B_j}\overline{z}}, e^{A_kz},e^{\overline{B_l}\overline{z}},1\}$ is linearly independent over the polynomials. This shows that
$$\sum_{i}\sum_{j}  a_i\overline{b_j} e^{A_iz+\overline{B_j}\overline{z}}\neq \sum_{i}\sum_{j}  a_i\overline{b_j}(v_{ij}- e^{A_iz+\overline{B_j}\overline{z}}).$$
Then $\{v_{ij}\}_{ij}$ is linearly independent.
This fact together with  (\ref{E28}) that
\begin{equation}\label{E37}
0=[e^{|z|^2}-q_m(|z|^2)]e^{A_iz+\overline{B_j}\overline{z}}-\rho_{ij}(z,\overline{z})
\end{equation}
 for each $i$ and $j$. Or equivalently,
\begin{equation}\label{E31}
\widetilde{e^{A_iz+\overline{B_j}\overline{z}}}=e^{A_iz+\overline{B_j}\overline{z}}
\end{equation}
for all $1\leq i \leq M_1$ and $1\leq j \leq M_2$ by (\ref{E37}) and the definition of the Berezin transform. So, (b) holds.

To prove (b)$\Rightarrow$(c). By lemma \ref{LL3} and (\ref{E31}), $A_i \overline{B_j}=0$ for all $1\leq i \leq M_1$ and $1\leq j \leq M_2.$ Assume there is a integer $\kappa$ so that $A_\kappa \overline{B_j}=0$ and $ A_\kappa \neq0,$ then we have
$B_j=0$ for $1\leq j \leq M_2.$ This shows that $g$ is a constant function. Thus, we have one of $f$ and $g$ is a constant function. The implication (c)$\Rightarrow$(a) is clear. This completes the proof.
\end{proof}

According to the proof of Proposition \ref{pp2}, we have that the Berezin transform will be useful if
$f$ and $g$ are \textit{exponential type} functions. In fact, we can rewrite $|K_m(z,w)|^2|w|^{2m}$ by
$$|K_m(z,w)|^2|w|^{2m}=\frac{(m!)^2}{|z|^{2m}}|e^{z\overline{w}}-q_m(z\overline{w})|^2.$$ Then we can almost get the form of $\widetilde{f\overline{g}}$ by the property of the reproducing kernel of Fock spaces if $f$ and $g$ are \textit{exponential type}.
 However,  this method obviously fails if $f(z)=K_m(z,w)$, see the proof of Lemma \ref{LL3}.
To solve this problem, we'll look at the Berezin
transform from a different angle. The idea comes from \cite{Qin}.

For any fixed $l\geq1,$ we define the expression
\begin{equation}\label{E32}
\rho(k)=\prod_{i=1}^l (k+m+i)=C_0+\sum_{i=1}^l C_i\bigg(\prod_{\lambda=0}^{i-1}(k-i)\bigg).
\end{equation}
Clearly,
$C_{l}=1$ and
\begin{equation}\label{EE9}
\rho(l)=\prod_{i=1}^l (l+m+i)=\sum_{i=0}^l C_i\frac{l!}{(l-i)!}.
\end{equation}
\begin{lemma}\label{LL1}
Let $A$ and  $B$ be two non-zero constants. For any fixed $l\geq1$, define the expression
\begin{align*}
C(A,B)= \sum_{k=0}^\infty \frac{m !}{k !} \frac{(k+l+m) !}{ (k+m)!} \overline{A}^{k} \overline{B}^{k} -\sum_{k=0}^{l} \frac{m !}{k !}   \frac{(l+m) !}{(k+m) !} \frac{(l+m) !}{(l-k+m) !} \overline{A}^{k} \overline{B}^{k}.
\end{align*}
Suppose $E=e^{\overline{AB}}=1$, then
\begin{align*}
C(A,B)&=\sum_{k=1}^{l-1} \frac{m !}{k !} \bigg\{\frac{(k+l+m) !}{ (k+m)!}- \frac{(l+m) !}{(k+m) !} \frac{(l+m) !}{(l-k+m) !} \bigg\} \overline{A}^{k} \overline{B}^{k}\\
&\quad -\sum_{i=0}^{l-2} m!C_i \bigg\{\sum_{k=1}^{l-i-1}\frac{1}{k!} \overline{A}^{k}
 \overline{B}^{k}\bigg\}
\overline{A}^{i} \overline{B}^{i}+m!d_l (\overline{AB})^l,
\end{align*}
where the $C_i$ is defined as before (\ref{E32}).
Moreover, $d_l\neq0$ if $m\neq0.$
\end{lemma}
\begin{proof}
 For simplicity, we write
$$D_1=\sum_{k\geq l+1} \frac{m !}{k !} \frac{(k+l+m) !}{ (k+m)!} \overline{A}^{k} \overline{B}^{k}, $$
 $$D_2=\sum_{k=1}^{l-1} \frac{m !}{k !} \bigg\{\frac{(k+l+m) !}{ (k+m)!}- \frac{(l+m) !}{(k+m) !} \frac{(l+m) !}{(l-k+m) !} \bigg\} \overline{A}^{k} \overline{B}^{k}$$
and $$D_3=\sum_{i=0}^{l-2} m!C_i \bigg\{\sum_{k=1}^{l-i-1}\frac{1}{k!} \overline{A}^{k}
 \overline{B}^{k}\bigg\}
\overline{A}^{i} \overline{B}^{i}.$$
It is easy to see that
$$C(A,B)=D_1+D_2+m!(\mathrm{C}_{2l+m}^{l}-\mathrm{C}_{l+m}^{l})\overline{A}^{l} \overline{B}^{l}$$
where $\mathrm{C}_{l+m}^{l}=\frac{(l+m)!}{m!{l!}}$.
 By (\ref{E32}) and a simple computation,
\begin{align*}
D_1&=\sum_{k\geq l+1}\frac{m !}{k !} \frac{(k+l+m) !}{ (k+m)!} \overline{A}^{k} \overline{B}^{k}\\
&=\sum_{k\geq l+1}\frac{m !}{k !} \bigg(C_0+\sum_{i=1}^l C_i\prod_{\lambda=0}^{i-1}(k-\lambda)\bigg) \overline{A}^{k} \overline{B}^{k}\\
&=\sum_{i=0}^l m!C_i\bigg\{ E-\sum_{k=0}^{l-i}\frac{1}{k!} \overline{A}^{k} \overline{B}^{k}\bigg\}
\overline{A}^{i} \overline{B}^{i}\\
&=\sum_{i=0}^{l-1} m!C_i\bigg\{ E-1-\sum_{k=1}^{l-i}\frac{1}{k!} \overline{A}^{k} \overline{B}^{k}\bigg\}
\overline{A}^{i} \overline{B}^{i}+m!C_l(E-1)\overline{A}^{l} \overline{B}^{l}.
\end{align*}
Since $E=1,$ it follows that
\begin{align*}
D_1&=\sum_{i=0}^{l-1} m!C_i\bigg\{-\sum_{k=1}^{l-i}\frac{1}{k!} \overline{A}^{k} \overline{B}^{k}\bigg\}
\overline{A}^{i} \overline{B}^{i}\\
&=\sum_{i=0}^{l-2} m!C_i \bigg\{-\sum_{k=1}^{l-i-1}\frac{1}{k!} \overline{A}^{k} \overline{B}^{k}\bigg\}
\overline{A}^{i} \overline{B}^{i}-\sum_{i=0}^{l-1} m!C_i \frac{1}{(l-i)!}\overline{A}^{l} \overline{B}^{l}.
\end{align*}
Let $d_l$ be coefficient  of $(\overline{AB})^l$, then $$C(A,B)=D_2-D_3+m!d_l\overline{A}^{l} \overline{B}^{l}.$$ Now, for $l\geq1,$ we consider the coefficient  of $(\overline{AB})^l$.  (\ref{EE9}) yields
\begin{align*}
d_l&=\mathrm{C}_{2l+m}^{l}-\mathrm{C}_{l+m}^{l}-\sum_{i=0}^{l-1} C_i \frac{1}{(l-i)!}\\
&=\mathrm{C}_{2l+m}^{l}-\mathrm{C}_{l+m}^{l}-\sum_{i=0}^{l} C_i \frac{1}{(l-i)!}+1\\
&=\mathrm{C}_{2l+m}^{l}-\mathrm{C}_{l+m}^{l}-\frac{1}{l!}\prod_{i=1}^l (l+m+i)+1\\
&=\frac{(2l+m)!}{l!(l+m)!}-\frac{(l+m)!}{l!m!}-\frac{(2l+m)!}{l!(l+m)!}+1\\
&=1-\frac{(l+m)!}{l!m!}.
\end{align*}
It is easy to check that $d_l\neq0$
if $m\neq0,$ which completes the proof.
\end{proof}
For $l\geq1,$ comparing the coefficient of $\overline{A}^l\overline{B}^l$ in Lemma \ref{LL1} and Lemma 4 in \cite{Qin}. $d_l$, i.e., the coefficient of $\overline{A}^l\overline{B}^l$,  is nonzero in Lemma \ref{LL1},  but $d_l=0$ is in \cite{Qin}. Combining these observations, we have the following result.
\begin{prop}\label{pp3}
Let $m$ be a positive integer. Suppose that $$f(z)=\sum_{i=1}^{N_1}a_iK_m(z,A_i),\quad g(z)=\sum_{\lambda=1}^{N_2}b_\lambda e^{zB_\lambda}.$$ Then
$T_{f\overline{g}}=T_{f}T_{\overline{g}}$ on $F^{2,m}$ if and only if at least one of $f$ and $g$ is constant function.
\end{prop}
\begin{proof}
It is clear that $T_f T_{\overline{g}}=T_{f\overline{g}}$ if one of $f$ and $g$ is constant function.
Conversely, assume  $T_f T_{\overline{g}}=T_{f\overline{g}}.$ Without loss of generality, we assume $a_i\neq 0$, $b_j\neq0,$  $\{A_i\}_{i=1}^{N_1}$ and $\{B_\lambda\}_{\lambda=1}^{N_2}$ are two  finite
collections of distinct non-zero complex numbers.

By the proof of Theorem 6 in \cite{Qin}, we can see that
\begin{align*}
& \sum_{\lambda=1}^{N_2}\sum_{i=1}^{N_1} a_{i} \overline{b_\lambda } \sum_{k_i=\max \{0, j-l\}} \frac{m !}{(k_i+m) !} \frac{(k_i+l+m) !}{(k_i+l-j) !} \overline{A_{i}}^{k_i} \overline{B_\lambda}^{k_i+l-j} \\
=&\sum_{\lambda=1}^{N_2}\sum_{i=1}^{N_1} a_{i}\overline{b_\lambda }\sum_{k_i=\max \{0, j-l\}}^{j} \frac{m !}{(k_i+m) !} \frac{(j+m) !}{(k_i+l-j) !} \frac{(l+m) !}{(j-k_i+m) !} \overline{A_{i}}^{k_i} \overline{B_\lambda}^{k_i+l-j}
\end{align*}
for all $j,l\geq0.$ Assume $0=l\leq j,$ one can see that
$$0=\sum_{\lambda=1}^{N_2}\sum_{i=1}^{N_1} a_{i}\overline{b_\lambda }  \sum_{k_i=j+1} \frac{1}{(k_i-j) !} \overline{A_{i}}^{k_i} \overline{B_\lambda}^{k_i-j}.$$ This shows that
\begin{equation}\label{E21}
0=\sum_{\lambda=1}^{N_2}\sum_{i=1}^{N_1} a_{i}\overline{b_\lambda } (e^{\overline{A_i} \overline{B_\lambda}}-1)A_i^{j}
\end{equation}
for all $j\geq 0.$ Using (\ref{E21}) and  the proof of  Theorem 6 in \cite{Qin}, it is easy to see that
\begin{equation}\label{E22}
e^{\overline{A_i }\overline{B_\lambda}}=1
\end{equation} for all $1\leq i \leq N_1,$ $1\leq \lambda \leq N_2.$

If $l=j=1,$ then we have
$$\sum_{\lambda=1}^{N_2}\sum_{i=1}^{N_1} a_{i} \overline{b_\lambda } \sum_{k_i=1} \frac{m !}{k_i!} (k_i+m+1)\overline{A_{i}}^{k_i} \overline{B_\lambda}^{k_i}=\sum_{\lambda=1}^{N_2}\sum_{i=1}^{N_1} a_{i} \overline{b_\lambda }(1+m)! \overline{A_{i}} \overline{B_\lambda}.$$
It follows that
\begin{align*}
&\quad\sum_{\lambda=1}^{N_2}\sum_{i=1}^{N_1} m!a_{i} \overline{b_\lambda }\bigg(e^{\overline{A_i} \overline{B_\lambda}}\overline{A_i} \overline{B_\lambda}-(m+1)(e^{\overline{A_i} \overline{B_\lambda}}-1)\bigg)\\
&=\sum_{\lambda=1}^{N_2}\sum_{i=1}^{N_1} a_{i} \overline{b_\lambda }(1+m)! \overline{A_{i}} \overline{B_\lambda}.
\end{align*}
Using (\ref{E22}), the above equation becomes
$$\sum_{\lambda=1}^{N_2}\sum_{i=1}^{N_1} m!a_{i} \overline{b_\lambda }\overline{A_i} \overline{B_\lambda}=\sum_{\lambda=1}^{N_2}\sum_{i=1}^{N_1} a_{i} \overline{b_\lambda }(1+m)! \overline{A_{i}} \overline{B_\lambda}.$$
This gives
\begin{equation}\label{E23}
\sum_{\lambda=1}^{N_2}\sum_{i=1}^{N_1}a_{i} \overline{b_\lambda } \overline{A_{i}} \overline{B_\lambda}=0.
\end{equation}

In fact, if we set $l=j\geq 1$, then
\begin{align}\label{E24}
& \sum_{\lambda=1}^{N_2}\sum_{i=1}^{N_1} a_{i} \overline{b_\lambda } \sum_{k_i=0}^\infty \frac{m !}{k_i !} \frac{(k_i+l+m) !}{ (k_i+m)!} \overline{A_{i}}^{k_i} \overline{B_\lambda}^{k_i} \nonumber\\
=&\sum_{\lambda=1}^{N_2}\sum_{i=1}^{N_1} a_{i}\overline{b_\lambda }\sum_{k_i=0}^{l} \frac{m !}{k_i !} \frac{(l+m) !}{(k_i+m) !} \frac{(l+m) !}{(j-k_i+m) !} \overline{A_{i}}^{k_i} \overline{B_\lambda}^{k_i}.
\end{align}
Setting \begin{align*}
C(A_i,B_\lambda)&= \sum_{\lambda=1}^{N_2}\sum_{i=1}^{N_1} a_{i} \overline{b_\lambda } \sum_{k_i=0}^\infty \frac{m !}{k_i !} \frac{(k_i+l+m) !}{ (k_i+m)!} \overline{A_{i}}^{k_i} \overline{B_\lambda}^{k_i}\\
&\quad- \sum_{\lambda=1}^{N_2}\sum_{i=1}^{N_1} a_{i}\overline{b_\lambda }\sum_{k_i=0}^{l} \frac{m !}{k_i !} \frac{(l+m) !}{(k_i+m) !} \frac{(l+m) !}{(j-k_i+m) !} \overline{A_{i}}^{k_i} \overline{B_\lambda}^{k_i}.
\end{align*}
By (\ref{E24}), we have
\begin{equation}\label{E26}
0=\sum_{\lambda=1}^{N_2}\sum_{i=1}^{N_1} a_{i} \overline{b_\lambda }C(A_i,B_\lambda).
\end{equation}

Assume $j=l=2.$ From Lemma \ref{LL1}, we obtain $$0=\sum_{\lambda=1}^{N_2}\sum_{i=1}^{N_1} a_{i} \overline{b_\lambda }\bigg\{\eta_1\overline{A_iB_\lambda}+\eta_2(\overline{A_iB_\lambda})^2\bigg\},$$
where
$\eta_k$ is the coefficient of $\overline{A}_i^k\overline{B_\lambda}^{k}$. Note that $\eta_k$ is independent
of $i$ and $\lambda.$ Moreover, $\eta_2\neq0$ by Lemma \ref{LL1}. Then, by (\ref{E23}), we obtain
\begin{equation}\label{E51}
0=\sum_{\lambda=1}^{N_2}\sum_{i=1}^{N_1} a_{i} \overline{b_\lambda } (\overline{A_iB_\lambda})^2.
\end{equation}

Comparing (\ref{E23}), (\ref{E26}), (\ref{E51}) and Lemma \ref{LL1},
one can see that
\begin{equation}\label{E25}
0=\sum_{\lambda=1}^{N_2}\sum_{i=1}^{N_1} a_{i} \overline{b_\lambda }\overline{A_i}^l \overline{B_\lambda}^l
\end{equation}
for all $l\geq1.$ By the "step 4" of proof of Theorem 6 in \cite{Qin}, (\ref{E25}) implies that one of the  $f$ and $g$ is constant function. This completes the proof.
\end{proof}
\begin{proofa}
Using the Theorem \ref{B}, Lemma \ref{LL6}, Proposition \ref{pp2} and \ref{pp3}, we get the Theorem.
\end{proofa}

\section{The symbol space in \cite{Bauer1}}
Recall that $$\mathcal{A}_1=\{f:f(z)=\sum_{i=1}^Np_i(z)e^{A_iz},\quad a_i,A_i\in \mathbb{C},\quad p_i\in \mathcal{P}\},$$ where
$\mathcal{P}$ is the algebra of holomorphic polynomials on $\mathbb{C}$.
In \cite{Bauer1}, the authors consider the semi-commuting Toeplitz operator with $\mathcal{A}_1-$symbol on Fock space. In this section, we try to characterize semi-commuting Toeplitz operators with $\mathcal{A}_1-$symbol on Fock-Sobolev space. We put
$$f^*(z)=\overline{f(\overline{z})}$$ for   holomorphic function $f$ on $\mathbb{C}.$ The following Lemma can be found in \cite{Bauer1}.
\begin{lemma}\cite{Bauer1}\label{LL5}
 The equality
\begin{align*}
\int_{\mathbb{C}} p(w)\overline{q(w)}e^{(A+\overline{z})w+(\overline{B}+z)\overline{w}-|w|^2}
\frac{dA(w)}{\pi}
=e^{(A+\overline{z})(\overline{B}+z)}q^*(\partial_z+\overline{z}+z)p(z+\overline{B}).
\end{align*}
holds for $A,B\in \mathbb{C},$ and $p,q\in \mathcal{P}.$
\end{lemma}

The proof of the second main result is divided into several steps. We now ready to prove the first part.
\begin{lemma}\label{LL4}
Suppose $f(z)=p(z)e^{Az}$, $g(z)=q(z)e^{Bz}$ and $p,q\in \mathcal{P} \setminus \{0\}$.
If $T_{f\overline{g}}=T_{f}T_{\overline{g}}$, then $e^{A\overline{B}}=1$ and the following statements
hold:\\
(a) If $A\neq0,$ then $q$ is constant.\\
(b) If $B\neq0,$ then $p$ is constant.
\end{lemma}
\begin{proof}
Since $T_{f\overline{g}}=T_{f}T_{\overline{g}}$, then
$$\widetilde{f\overline{g}}(z)=p(z)\overline{q(z)}e^{Az+\overline{Bz}}.$$
We compute the Berezin transform of $f\overline{g}$
\begin{align*}
\widetilde{f\overline{g}}(z)&=  \frac{1}{m!\pi}\int_{\mathbb{C}} p(w)\overline{q(w)}e^{Aw+\overline{Bw}}|k_m(w,z)|^2|w|^{2m}e^{-|w|^2}dA(w)\\
&=\frac{m!}{K_m(z,z)}  \frac{1}{\pi}\int_{\mathbb{C}} p(w)\overline{q(w)}e^{Aw+\overline{Bw}}\frac{|e^{z\overline{w}}
-q_m(z\overline{w})|^2}{|z|^{2m}}e^{-|w|^2}dA(w).
\end{align*}
So we have
\begin{align*}
&\quad \frac{|z|^{2m}K_m(z,z)}{m!}\widetilde{f\overline{g}}(z)\\
&=  \frac{1}{\pi}\int_{\mathbb{C}} p(w) \overline{q(w)}e^{Aw+\overline{Bw}}{|e^{z\overline{w}}
-q_m(z\overline{w})|^2}e^{-|w|^2}dA(w).
\end{align*}
For simplicity, we write
$$Q_1=  \frac{1}{\pi}\int_{\mathbb{C}} p(w) \overline{q(w)}e^{Aw+\overline{Bw}}e^{\overline{z}w+z\overline{w}}e^{-|w|^2}dA(w),$$
$$ Q_2= \frac{1}{\pi}\int_{\mathbb{C}} p(w) \overline{q(w)}e^{Aw+\overline{Bw}}
\bigg(|q_m(\overline{z}w)|^2-q_m(\overline{w}z)e^{w\overline{z}}\bigg)e^{-|w|^2}dA(w),$$
and
$$Q_3=  \frac{1}{\pi}\int_{\mathbb{C}} p(w)\overline{q(w)}q_m(w\overline{z})e^{Aw+\overline{Bw}}e^{\overline{w}z}e^{-|w|^2}dA(w).$$
Then $$\frac{|z|^{2m}K_m(z,z)}{m!}\widetilde{f\overline{g}}(z)=Q_1+Q_2-Q_3.$$
Note that the "coefficient" of  $e^{|z|^2}$ in $Q_2$ and $Q_3$ are both zero.

Using Lemma \ref{LL5}, we can obtain
\begin{equation}\label{E34}
\frac{|z|^{2m}K_m(z,z)}{m!}\widetilde{f\overline{g}}(z)=
e^{(A+\overline{z})(\overline{B}+z)}q^*(\partial_z+\overline{z}+z)p(z+\overline{B})
+Q_2-Q_3.
\end{equation}
Since $\widetilde{f\overline{g}}(z)=p(z)\overline{q(z)}e^{Az+\overline{Bz}},$ then, by above equation, \begin{align}\label{E35}
 [e^{|z|^2}-q_m(|z|^2)]\widetilde{f\overline{g}}(z)&=
p(z)\overline{q(z)}[e^{|z|^2}-q_m(|z|^2)]e^{Az+\overline{Bz}}\nonumber\\
&=e^{(A+\overline{z})(\overline{B}+z)}q^*(\partial_z+\overline{z}+z)p(z+\overline{B})+Q_2-Q_3.
\end{align}
Consider the "coefficient" of $e^{|z|^2}$ in (\ref{E35}), one can see that
\begin{equation}\label{E33}
p(z)\overline{q(z)}e^{Az+\overline{Bz}}=e^{A\overline{B}+Az+\overline{B}
\overline{z}}q^*(\partial_z+\overline{z}+z)p(z+\overline{B}).
\end{equation}
By (\ref{E33}) and the proof of Lemma 3.4 in \cite{Bauer1}, $e^{A\overline{B}}=1.$ Moreover, if $B\neq0,$ then $p$ is constant. This shows (b). On the other hand, $q$ is constant if $A\neq0$ via complex conjugation. This completes the proof.
\end{proof}
\begin{remark}\label{Re1}
Note that
 $$|q_m(z)|=\sqrt{\sum_{i,j}\frac{z^i\overline{z}^j}{i!j!}}\leq \sqrt{\sum_{i,j}|z|^{i+j}}\lesssim \sqrt{1+|z|^{2m}}\lesssim 1+|z|^m.$$
The above fact and a simple computation show that
\begin{align*}
|Q_3|& \lesssim  \int_{\mathbb{C}} |p(w)q(w)e^{(A+B)w}||q_m(\overline{w}z)e^{\overline{w}z}|e^{-|w|^2}dA(w)\\
&\lesssim \int_{\mathbb{C}} |q(w)p(w)e^{(A+B)w}||(1+|wz|^m)e^{\overline{w}z}|e^{-|w|^2}dA(w)\\
&\lesssim(1+|z|^m)\int_{\mathbb{C}/\{0\}} |w^mp(w)q(w)e^{(A+B)w}||e^{\overline{w}z}|e^{-|w|^2}dA(w)\\
&\lesssim (1+|z|^m) \int_{\mathbb{C}} |w^mp(w)q(w)e^{(A+B)w}||e^{\overline{w}z}|e^{-|w|^2}dA(w).
\end{align*}
By Cauchy-Bunyakovsky-Schwarz Inequality,
$$|Q_3| \lesssim (1+|z|^m)\|e^{w\overline{z}}\|_2\lesssim (1+|z|^m) e^{\frac{|z|^2}{2}}\lesssim e^{\frac{|z|^2}{2}},$$
here $\|\cdot\|_2$ is the norm of $L^2(\mathbb{C},e^{-|z|^2}dA(z))$.
For $z\neq0,$ it follows that
$$\frac{|Q_3|}{e^{|z|^2}-q_m(|z|^2)}\lesssim \frac{e^{\frac{|z|^2}{2}}}{{e^{|z|^2}-q_m(|z|^2)}}<\infty.$$
The same inequality holds with $Q_2$ in place of $Q_3$. We thus obtain
$\frac{|Q_2|}{e^{|z|^2}-q_m(|z|^2)}$ and $\frac{|Q_3|}{e^{|z|^2}-q_m(|z|^2)}$ are both bounded on $\mathbb{C}/\{0\}.$

\end{remark}
 The above Lemma shows the Berezin transform of a polynomial is again a polynomial by the case $A=b=0$. From Lemma \ref{LL3} and \ref{LL4}, we get the following result.
\begin{prop}\label{pp1}
Suppose $f(z)=p(z)e^{Az}$, $g(z)=q(z)e^{Bz}$ and $p,q\in \mathcal{P} \setminus \{0\}$.
Then $T_{f\overline{g}}=T_{f}T_{\overline{g}}$ if and only if one of $f$ and $g$ is a constant function.
\end{prop}
\begin{proof}
By Lemma  \ref{LL4}, we have $g(z)=e^{Bz}$ if $A\neq0$. Similarly, $f(z)=e^{Az}$ if $B\neq0.$
In this case, $T_{f\overline{g}}=T_{f}T_{\overline{g}}$ implies one of  $f$ and $g$ is constant
by Lemma \ref{LL3}.

Now, assume $A=0.$ Then $f$  is a constant if $B\neq0.$ So we assume $B=0$, it follows that
$f$ and $g$ are holomorphic polynomials. So,  $T_{f\overline{g}}=T_{f}T_{\overline{g}}$ gives
one of $f$ and $g$ is a constant by Theorem 9 in \cite{Bauer1}.

On the other hand, $T_{f\overline{g}}=T_{f}T_{\overline{g}}$ if one of $f$ and $g$ is a constant function. This completes the proof.
\end{proof}
Now, we can complete  characterize semi-commuting Toeplitz operators with $A_1$-symbol on Fock-Sobolev space.
\begin{theorem}\label{TT2}
Suppose $f,g\in \mathcal{A}_1$ and $m$ is a positive integer.   Then $T_{f\overline{g}}=T_{f}T_{\overline{g}}$ on $F^{2,m}$ if and only if
at least one of $f$ and $g$ is a constant.
\end{theorem}
\begin{proof}
The sufficient direction is clear. To prove the necessary direction.
Using the following notations
$$R_{A,B,p,q}(z,\overline{z})=e^{A\overline{B}
}q^*(\partial_z+\overline{z}+z)p(z+\overline{B}).$$
By (\ref{E34}) , (\ref{E35}), and the proof of Lemma \ref{LL4}, one can see that
\begin{align*}
&\quad \sum_{i,k} p_i(z)\overline{q_k(z)}[e^{|z|^2}-q_m(|z|^2)]e^{A_iz+\overline{B_kz}}\nonumber\\
&= \sum_{i,k} e^{(A_i+\overline{z})(\overline{B_k}+z)}q^*_k(\partial_z+\overline{z}+z)p(z+\overline{B_k})+Q_2^{ik}-Q_3^{ik},
\end{align*}
where $Q_2^{ik}$ and  $Q_3^{ik}$ as defined in (\ref{E35}). Recall that the "coefficient" of  $e^{-|z|^2}$ in $Q_2^{ik}$ and $Q_3^{ik}$ are both zero.  It follows that
$$\sum_{i,k} p_i(z)\overline{q_k(z)}e^{A_iz+\overline{B_kz}}=\sum_{i,k} e^{A_i\overline{B_k}+A_iz+\overline{B_k}
\overline{z}}q_k^*(\partial_z+\overline{z}+z)p(z+\overline{B_k}),$$
or equivalently,
$$\sum_{i,j}\bigg[R_{A_i,B_j,p_i,q_k}(z,\overline{z})-p_i(z)\overline{q_k}(z)\bigg]e^{A_iz+\overline{B_jz}}=0.$$
By the proof of Theorem 3.5 in \cite{Bauer1} again, we can obtain
$$0=R_{A_i,B_j,p_i,q_k}(z,\overline{z})-p_i(z)\overline{q_k}(z)$$
for each $i$ and $j$.
Form the proof of the Lemma 3.4 in \cite{Bauer1}, we have $e^{A_i\overline{B_k}}=1$ for each $i$ and $k$. Moreover, the following statements hold:
(a) If there is a  $A_i\neq 0,$ then each $q_k$ is a constant;
(b) If there is a $B_k\neq0,$ then each $p_i$ is constant.
By Proposition \ref{pp2} and \ref{pp1},
we obtain the result.
\end{proof}
If we set $m=0$,  we will obtain Theorem \ref{A} by the proof of Proposition \ref{pp1} and  Theorem \ref{TT2}. We remark that Theorem \ref{TT2} can be extend to the Fock-Sobolev spaces on $\mathbb{C}^n$ by Berezin transform.
\begin{proofb}
Theorem B comes follow Theorem \ref{TT2}.
\end{proofb}
\section{Bounded Semi-Commutants }
In this section, we will consider the bounded semi-commutants of Toeplitz operators on Fock-Sobolev spaces. In the rest of the paper, the $\langle \cdot,\cdot\rangle_{F^2}$ denote the inner product on Fock space, and $\|\cdot\|_{2,m}$ denote the norm of Fock-Sobolev space. Recall that $a\in 2 \mathbb{N} \pi \mathrm{i}$ means that
$a= 2n\pi \mathrm{i}$ for any integer $n.$
\begin{lemma}\label{pp4}
Suppose  $A$ and $B$ are  non-zero complex numbers. If $A\overline{B}\in 2 \mathbb{N} \pi \mathrm{i}/\{0\}$, then $(T_{e^{Az}},T_{e^{\overline{Bz}}}]$ is bounded on $F^{2,m}$.
\end{lemma}
\begin{proof}
For any $f\in F^{2,m},$ we have
$$T_{e^{\overline{Bz}}}f(z)=\frac{1}{\pi m!}\int_{\mathbb{C}} e^{\overline{Bw}}f(w)K(z,w)|w|^{2m}e^{-|w|^2}dA(w)$$
and
 $$T_{e^{Az+\overline{Bz}}}f(z)=\frac{1}{\pi m!}\int_{\mathbb{C}} e^{Aw+\overline{Bw}}f(w)K(z,w)|w|^{2m}e^{-|w|^2}dA(w).$$
Let $z\neq 0,$ it follows that
\begin{align*}
T_{e^{\overline{Bz}}}f(z)&=\frac{1}{z^m\pi }\int_{\mathbb{C}} w^me^{\overline{Bw}}f(w)[{e^{z\overline{w}}-q_m(z\overline{w})}]e^{-|w|^2}dA(w)\\
&=\frac{1}{z^m}\langle w^mf(w),e^{w\overline{z}+Bw}\rangle_{F_2}-\frac{1}{z^m} \langle w^mf(w),q_m(\overline{z}w)e^{Bw}\rangle_{F^2}\\
&=\frac{1}{z^m} (z+\overline{B})^mf(z+\overline{B})-S_1(z),
\end{align*}
where $S_1(z)=\frac{1}{z^m} \langle w^mf(w),q_m(\overline{z}w)e^{Bw}\rangle_{F^2}.$ We also have
\begin{align*}
T_{e^{Az+\overline{Bz}}}f(z)&=\frac{1}{z^m\pi }\int_{\mathbb{C}} w^me^{Aw+\overline{Bw}}f(w)[{e^{z\overline{w}}-q_m(z\overline{w})}]e^{-|w|^2}dA(w)\\
&=\frac{1}{z^m}\langle w^mf(w)e^{Aw},e^{w\overline{z}+Bw}\rangle_{F^2}-\frac{1}{z^m} \langle w^mf(w)e^{Aw},q_m(\overline{z}w)e^{Bw}\rangle_{F^2}\\
&=\frac{1}{z^m} (z+\overline{B})^mf(z+\overline{B})e^{Az+A\overline{B}}-S_2(z),
\end{align*}
where $S_2=\frac{1}{z^m} \langle w^mf(w)e^{Aw},q_m(\overline{z}w)e^{Bw}\rangle_{F^2}.$ Note that $e^{A\overline{B}}=1.$ Then, from the above
equation, we obtain
\begin{equation}\label{E46}
T_{e^{Az+\overline{Bz}}}f(z)-T_{e^{Az}}T_{e^{\overline{Bz}}}f(z)=e^{Az}S_1(z)-S_2(z).
\end{equation}
 A careful calculation gives
\begin{align*}
|e^{Az}S_1(z)-S_2(z)|&\leq \frac{1}{\pi|z|^m}\int_{\mathbb{C}}|w^mf(w)||e^{Bw}||e^{Az}-e^{Aw}||q_m(z\overline{w})|w^{-|w|^2}dA(w)\nonumber\\
&\lesssim \int_{\mathbb{C}}|w^mf(w)||e^{Bw}||e^{Az}-e^{Aw}|\frac{1+|zw|^m}{|z|^m}e^{-|w|^2}dA(w)\nonumber\\
&\lesssim \int_{\mathbb{C}/\{0\}}|w^mf(w)||e^{Bw}||e^{Az}-e^{Aw}|\frac{|zw|^m}{|z|^m}e^{-|w|^2}dA(w)\nonumber
\end{align*}
So we have
\begin{align}\label{E45}
|e^{Az}S_1(z)-S_2(z)|
&\lesssim \int_{\mathbb{C}} |f(w)||e^{Bw}||e^{Az}-e^{Aw}||w|^{2m}e^{-|w|^2}dA(w).
\end{align}
By (\ref{E45}) and Cauchy-Bunyakovsky-Schwarz Inequality, one can see that
\begin{align*}
|e^{Az}S_1(z)-S_2(z)|&\lesssim \|f\|_{2,m}\|e^{Bw}(e^{Az}-e^{Aw})\|_{2,m}\nonumber\\
&\leq \|f\|_{2,m} (|e^{Az}|\|e^{Bw}\|_{2,m}+\|e^{(A+B)w}\|_{2,m})\nonumber\\
&\lesssim |e^{Az}| \|f\|_{2,m}.
\end{align*}
This together with (\ref{E46}) shows that
\begin{align*}
\|T_{e^{Az+\overline{Bz}}}f-T_{e^{Az}}T_{e^{\overline{Bz}}}f(z)\|^2_{2,m}
&=\frac{1}{m!\pi}\int_{\mathbb{C}}|e^{Az}S_1(z)-S_2(z)|^2|z|^{2m}e^{-|z|^2}dA(z)\\
&\lesssim    \frac{\|f\|_{2,m}^2}{m!\pi}\int_{\mathbb{C}}|e^{Az}|^2|z|^{2m}e^{-|z|^2}dA(z)\\
&= \|f\|_{2,m}^2  \|e^{Az}\|_{2,m}^2,
\end{align*}
which implies $T_{e^{Az+\overline{Bz}}}-T_{e^{Az}}T_{e^{\overline{Bz}}}$ is bounded on $F^{2,m}.$ The proof is complete.
\end{proof}
For the linear exponential functions, we have the follow result for the bounded semi-commutants.
\begin{prop}\label{pp5}
Suppose $f(z)=e^{Az}$ and $g(z)=e^{Bz}$. Then $(T_{f},T_{\overline{g}}]$ is bounded if and only if one of the following hold:\\
(a) one of $f$ and $g$ is a constant.\\
(b) $A\overline{B}\in2 \mathbb{N} \pi \mathrm{i}$.\\
(c) $A=-B.$
\end{prop}
\begin{proof}
The sufficiency follows from Lemma \ref{pp4} and the main result of \cite{Chen}.
Now, assume $(T_{f},T_{\overline{g}}]$ is bounded and $AB\neq0$, then
$$\widetilde{(T_{f},T_{\overline{g}}]}(z)=\widetilde{f\overline{g}}(z)-f(z)\overline{g}(z)$$
is also bounded. We compute the Berezin transform
\begin{equation}\label{E39}
\widetilde{f\overline{g}}(z)-f(z)\overline{g}(z)=\frac{|z|^{2m}}
{m![e^{|z|^2}-q_m(|z|^2)]}Q(z,\overline{z})-e^{Az+\overline{B}\overline{z}},
\end{equation}
where $Q(z,\overline{z})$ is defined as before in Lemma \ref{LL3}.  Here,
\begin{align*}
\frac{|z|^{2m}}{m!}Q(z,\overline{z})&=e^{(A+\overline{z})
(\overline{B}+z)}-e^{\overline{B}(A+\overline{z})}
q_m(zA+|z|^2)\nonumber\\
&\quad-e^{A(\overline{B}+z)}
q_m(\overline{z}\overline{B}+|z|^2)+p(z,\overline{z}).
\end{align*}
If $z=0$, then $\widetilde{(T_{f},T_{\overline{g}}]}(0)$ is bounded.
So, we always assume $z\neq0.$

It follows that
\begin{align}\label{E40}
\frac{|z|^{2m}}
{m![e^{|z|^2}-q_m(|z|^2)]}Q(z,\overline{z})&=\frac{e^{(A+\overline{z})
(\overline{B}+z)}-e^{\overline{B}(A+\overline{z})}
q_m(zA+|z|^2)}{e^{|z|^2}-q_m(|z|^2)}\\
&\quad+\frac{p(z,\overline{z})-e^{A(\overline{B}+z)}
q_m(\overline{z}\overline{B}+|z|^2)}{e^{|z|^2}-q_m(|z|^2)}\nonumber.
\end{align}
Note
\begin{equation}\label{E41}
G(A,B)(z)=\frac{p(z,\overline{z})-e^{\overline{B}(A+\overline{z})}
q_m(zA+|z|^2)-e^{A(\overline{B}+z)}
q_m(\overline{z}\overline{B}+|z|^2)}{e^{|z|^2}-q_m(|z|^2)}.
\end{equation}
Clearly, $G(A,B)$  is bounded for $z\neq0$.

It is clear that
\begin{align*}
|\widetilde{(T_{f},T_{\overline{g}}]}|\geq|\widetilde{f\overline{g}}(z)|
-|f(z)\overline{g}(z)|.
\end{align*}
By (\ref{E40}), (\ref{E41} )and  the boundedness of $G(A,B)(z)$,
\begin{align*}
|\widetilde{(T_{f},T_{\overline{g}}]}|&
\geq\frac{|e^{(A+\overline{z})
(\overline{B}+z)}|}{e^{|z|^2}-q_m(|z|^2)}
-|e^{Az+\overline{Bz}}|-c\\
&=\frac{|e^{|z|^2+A\overline{B}+Az+\overline{Bz}}|}{e^{|z|^2}-q_m(|z|^2)}-|e^{Az+\overline{Bz}}|-c\\
&> \frac{|e^{|z|^2+A\overline{B}+Az+\overline{Bz}}|}{e^{|z|^2}}-|e^{Az+\overline{Bz}}|-c\\
&=|e^{A\overline{B}+Az+\overline{Bz}}|-|e^{Az+\overline{Bz}}|-c\\
&=|e^{A\overline{B}}-1||e^{Az+\overline{Bz}}|-c\\
&=|e^{A\overline{B}}-1||e^{(A+B)z}|-c
\end{align*}
where $c$ is a nonnegative constant. Since $\widetilde{(T_{f},T_{\overline{g}}]}(z)$ is bounded, then
\begin{equation*}
|(e^{A\overline{B}}-1)e^{(A+B)z}|<\infty.
\end{equation*}
This shows that $e^{A\overline{B}}=1$ or $A=-B.$ This completes the proof.
\end{proof}
 As an immediate consequence of Proposition \ref{pp5}, we obtain the following theorem, which recover the result in \cite{MA}.
\begin{theorem}\label{TT3}
Suppose $f(z)=p(z)e^{Az}$, $g(z)=q(z)e^{Bz}$ and $p,q\in \mathcal{P} \setminus \{0\}$.
Then $(T_f,T_{\overline{g}}]$ is bounded  if and only if one of  following holds.\\
(a) one of $f$ and $g$ is constant.\\
(b) $f(z)=e^{Az}$ and $g(z)=e^{Bz}$, where $A\overline{B}\in2\mathbb{N}\pi \mathrm{i}$.\\
(c) $f(z)=e^{Az}$ and $g(z)=e^{Bz}$ with $A=-B.$\\
(d) $f$ and $g$ are linear holomorphic polynomials.
\end{theorem}
\begin{proof}
The sufficiency is clear. We now prove the necessity.
Since $(T_f,T_{\overline{g}}]$ is bounded, then $\widetilde{(T_f,T_{\overline{g}}]}$ is bounded on $\mathbb{C}$. Namely,
$$|\widetilde{f\overline{g}}(z)-f(z)\overline{g(z)}|<\infty.$$

Now, we assume $z\neq0$, since $|\widetilde{(T_f,T_{\overline{g}}]}(0)|<\infty$.
It follows from (\ref{E34}) that
\begin{equation}\label{EEE1}
\bigg|\frac{e^{(A+\overline{z})(\overline{B}+z)}q^*(\partial_z+\overline{z}+z)p(z+\overline{B})
+Q_2-Q_3}{e^{|z|^2}-q_m(|z|^2)}-p(z)\overline{q(z)}e^{Az+\overline{Bz}}\bigg|<\infty,
\end{equation}
where $Q_2$ and  $Q_3$ as defined in (\ref{E35}). Using Remark \ref{Re1} and (\ref{EEE1}),
$$\bigg|\frac{e^{(A+\overline{z})(
\overline{B}+z)}q^*(\partial_z+\overline{z}+z)p(z+\overline{B})
-p(z)\overline{q(z)}[e^{|z|^2}-q_m(|z|^2)]e^{Az+\overline{Bz}}}{e^{|z|^2}-q_m(|z|^2)}\bigg|<\infty.$$
This implies
\begin{equation}\label{EE7}
|e^{Az+Bz}|\bigg|e^{A\overline{B}}q^*(\partial_z+\overline{z}+z)p(z+\overline{B})
-p(z)\overline{q(z)}\bigg|<\infty.
\end{equation}

If $A=B=0,$ then $f$ and $g$ are linear holomorphic polynomials by easily
modifying the proof of Theorem 19 in \cite{MA}.

If $A=0$ and $B\neq0$, then, by (\ref{EE7}),
$$|e^{Bz}|\bigg|q^*(\partial_z+\overline{z}+z)p(z+\overline{B})
-p(z)\overline{q(z)}\bigg|<\infty.$$
Note that $q^{*}$ is a polynomial in $z$ and $\overline{z}$. Then above equation shows that $$q^*(\partial_z+\overline{z}+z)p(z+\overline{B})
=p(z)\overline{q(z)}.$$
By the proof of Theorem \ref{TT2}, $p$ is nonzero constant. In this case,
$f(z)=f(0)$ and $g(z)=e^{Bz}$.

 In fact,  $A\overline{B}\neq0$ and $e^{Az+Bz}$ is unbounded on $\mathbb{C}$ claim that $$q^*(\partial_z+\overline{z}+z)p(z+\overline{B})-p(z)\overline{q(z)}=0.$$ The proof of Theorem \ref{TT2} shows that $p$ and $q$ must be constants and $e^{A\overline{B}}=1$. On the other hand, $e^{Az+Bz}$ is bounded together  with $A\overline{B}\neq0$ gives $$e^{A\overline{B}}q^*(\partial_z+\overline{z}+z)p(z+\overline{B})-p(z)\overline{q(z)}$$ is bounded on $\mathbb{C}$. By the proof of Lemma 3.4 in \cite{Bauer1}, we have $p$ and $q$ must be constants.
 To sum up, then $p$ and $q$ must be constants if $A\overline{B}\neq0$. In this case, we have $A=-B$ or $e^{A\overline{B}}=1$ by Proposition \ref{pp5}. This completes the proof.
\end{proof}

 Recall that
$$
D(f, g)(z)=\left[\widetilde{|f|^{2}}(z)-|\widetilde{f}(z)|^{2}\right]\left[\widetilde{|g|^{2}}(z)-|\widetilde{g}(z)|^{2}\right]
.$$ Now, we consider the conjecture about the Hankel products. We will show it is false on $F^{2,m}.$
\begin{theorem}
Let $f(z)=e^{2\pi\mathrm{ i}z }$ and $g(z)=e^z$. Then $(T_f,T_{\overline{g}}]=H_{\overline{f}}^*H_{\overline{g}}$ is bounded but $D(f,g)$ is unbounded on $\mathbb{C}$.
\end{theorem}
\begin{proof}
The boundedness of $(T_f,T_{\overline{g}}]$ follows from Theorem \ref{TT3}.
By Lemma 2.4 in \cite{Qin1}, there is a positive constant $\varepsilon$ independent of $z$ and $w$ such that
\begin{align*}
\|H_{\overline{f}}k_{m,z}\|^2&=\widetilde{|f|^2}(z)-|\widetilde{f}(z)|^2\\
&\gtrsim \int_{|z-w|<\varepsilon}|f(w)-f(z)|^2dA(w)
\end{align*}
for sufficiently large $|z|.$
It follows that
\begin{align*}
\|H_{\overline{f}}k_{m,z}\|^2&\gtrsim
\int_{|z-w|<\varepsilon}|f(w)-f(z)|^2dA(w)\\
&=\int_{|z-w|<\varepsilon} |e^{2\pi\mathrm{ i}(w-z)}-1|^2dA(w)\\
&=|e^{2\pi\mathrm{ i}z}|^2\int_{|w|<\varepsilon} |e^{2\pi\mathrm{ i}w}-1|^2dA(w)\\
&\gtrsim |e^{2\pi\mathrm{ i}z}|^2,
\end{align*}
since $0<\int_{|w|<\varepsilon} |e^{2\pi\mathrm{ i}w}-1|^2dv(w)<\infty.$ We also have
$$\|H_{\overline{g}}k_{m,z}\|^2\gtrsim |e^z|^2 .$$ Thus we have
$$\sup_{z\in \mathbb{C}} D(f,g)(z)\gtrsim \sup_{z\in \mathbb{C}} |e^{(2\pi\mathrm{ i}+1)z}|^2=\infty .$$ This proves the desired.
\end{proof}
\section{Open problems and Remarks}
In this work, we consider the semi-commutant of Toeplitz operators on Fock-Sobolev spaces.  The
main results show that there is the fundamental difference between the geometries of Fock and
Fock-Sobolev space. This result can be generalized to some weighted Fock spaces, whose reproducing kernel function can be decomposed into the kernel of the Fock space. We remark the results in Section 3 and 4 can be extend to the Fock-Sobolev spaces on $\mathbb{C}^n$.

 Note that
$$K_m(z,z)\simeq \frac{e^{|z|^2}}{1+|z|^{2m}}.$$ Combining these observation, we conclude that the properties of operator  on Fock-Sobolev space are the same as
it's properties on Fock space if we only use the estimate of reproducing kernel $K_m(z,z)$. For instance, Sarason's problem,  the boundedness and compactness of composition operators, the commuting Toeplitz operators with radial symbols, Ha-plitz product and so on, see \cite{Bauer2,Carswell,Chen,Cho2,Cho1,Choe,MA1,Qin1}. However, there is a fundamental difference between the geometries of Fock and
Fock-Sobolev space if we use the specific form of $K_m(z,w),$ see Theorem \ref{A}, Theorem A and Theorem B. So, it would be interesting to consider the following problem.  Before we
give this problem, we establish the following symbol space. Setting $m\neq0$, we define the symbol space $\mathcal{A}$:
$$\mathcal{A}=\{f:\sum_{i=1}^N p_i(z)K_m(z,A_i),\quad p_i\in \mathcal{P}\}.$$
Since $e^{\overline{a}z}=(\overline{a}z)^mK_m(z,a)+q_m(\overline{a}z),$ one can see that
$$\mathcal{D}\subset \mathcal{A},\quad \mathcal{A}_1\subseteq \mathcal{A} ,\quad \mathcal{A}\nsubseteq \mathcal{A}_1.$$
\begin{proc}
Let $m\neq0$, and suppose $f,g\in\mathcal{A}.$ What is the relationship between their symbols when two Toeplitz operators $T_f$ and $T_{\overline{g}}$
commute on $F^{2,m}?$
\end{proc}

Originally, we want to consider the semi-commutant of Toeplitz operators with $\mathcal{A}-$symbol. We
have no idea of how to solve this problem. The reason is that we don't even know the specific form of Berezin transform of $K_m(z,a)\overline{K_m(z,b)}$ on Fock-Sobolev space. We believe the essential difference is that the Berezin transform on Fock space can be written out precisely, but on
Fock-Sobolev  space little is known about the Berezin transform of $K_m(z,a)\overline{K_m(z,b)}$.

Another interesting problem is the Brown-Halmos Theorems on Fock-Sobolev spaces, see \cite{Brown,Le} for more information. In \cite{Bauer1}, the authors have considered the zero product problem on Fock space under the hypothesis $T_uT_v=T_vT_u=0$, here  $f,g,h,k\in \mathcal{A}_1$, $u=f+\overline{k}$  and $v=h+\overline{g}.$  We believe the hypothesis can be relaxed to $T_uT_v=0$, since $(T_f,T_{\overline{g}}]=(T_h,T_{\overline{k}}]=0$ only have trivial
solutions if $m\neq0.$ So, there is a  natural conjecture.
\begin{conj}
For $f,g,h,k\in \mathcal{A}_1$ put $u=f+\overline{k}$  and $v=h+\overline{g}.$ If $T_uT_v=0$ on $F^{2,m}$ and $m\neq0$, then either $u=0$ or $v=0$.
\end{conj}

We hope this work will generate some further interest in Operator Theory on Sobolev spaces.

\providecommand{\MR}{\relax\ifhmode\unskip\space\fi MR }
\providecommand{\MRhref}[2]{%
  \href{http://www.ams.org/mathscinet-getitem?mr=#1}{#2}
}
\providecommand{\href}[2]{#2}

\end{document}